\documentclass[10pt,portrait,reqno,12pt]{amsart}%
\usepackage{amssymb,amsthm,amsmath, amsfonts, float}
\usepackage{graphicx}
\usepackage{multirow}
\usepackage{subfig}
\usepackage{amsmath}
\usepackage{amsfonts}
\usepackage{lscape}
\usepackage{amssymb}%
\providecommand{\U}[1]{\protect\rule{.1in}{.1in}}
\newtheorem{theorem}{Theorem}
\theoremstyle{plain}

\newtheorem{corollary}{Corollary}

\newtheorem{example}{Example}

\newtheorem{proposition}{Proposition}

\numberwithin{equation}{section}
\begin{document}
\title[Timelike General Rotational Surfaces in $\mathbb{E}_{1}^{4}$ with Density]{Timelike General Rotational Surfaces in Minkowski 4-Space with Density}
\subjclass[2010]{53A10, 53A35, 53B30.}
\keywords{Timelike general rotational surface, density, weighted mean curvature,
weighted Gaussian curvature.}
\author[A. Kazan, M. Alt\i n and D.W. Yoon]{\bfseries Ahmet Kazan$^{1\ast}$, Mustafa Alt\i n$^{2}$ and Dae Won Yoon$^{3}$}
\address{ \newline
$^{1}$\textit{Department of Computer Technologies, Do\u{g}an\c{s}ehir Vahap
K\"{u}\c{c}\"{u}k Vocational School, Malatya Turgut \"{O}zal University,
Malatya, Turkey} \newline
$^{2}$\textit{Technical Sciences Vocational School, Bing\"{o}l University,
Bing\"{o}l, Turkey} \newline
$^{3}$\textit{Department of Mathematics Education and RINS, Gyeongsang
National University, Jinju 52828, Republic of Korea} \newline
$^{\ast}$\textit{Corresponding author: ahmet.kazan@ozal.edu.tr}}

\begin{abstract}
In this study, we give weighted mean and weighted Gaussian curvatures of two
types of timelike general rotational surfaces with non-null plane meridian
curves in four-dimensional Minkowski space $\mathbb{E}_{1}^{4}$ with density
$e^{\lambda_{1}x^{2}+\lambda_{2}y^{2}+\lambda_{3}z^{2}+\lambda_{4}t^{2}},$
where $\lambda_{i}$ ($i=1,2,3,4$) are not all zero and we give some results
about weighted minimal and weighted flat timelike general rotational surfaces
in $\mathbb{E}_{1}^{4}$ with density. Also, we construct some examples for
these surfaces.

\end{abstract}
\maketitle


\section{\textbf{INTRODUCTION}}

A four-dimensional space (4D) can be thought of as a geometrical extension of
the three-dimensional space (3D). So, the 4D spaces have become one of the
fundamental concepts in expressing modern physics, mathematics, and geometry.
In this context, lots of studies about curves and (hyper)surfaces in 4D
Euclidean, Minkowskian, Galilean, or pseudo-Galilean spaces have been done by
geometers, recently. In particular, general rotational surfaces in the 4D
Euclidean space were introduced by Cole \cite{Cole} and later studied by Moore
\cite{Moore}. After these studies, different characterizations of general
rotational surfaces were interested in many geometers. For instance, Ganchev
and Milousheva studied minimal surfaces and they obtained the necessary and
sufficient condition for the minimality of general rotational surfaces in
$\mathbb{E}^{4}$ \cite{Ganchev}. Dursun and Turgay examined general rotational
surfaces in $\mathbb{E}^{4}$ whose meridian curves lie in two-dimensional
planes and they found all minimal general rotational surfaces by solving the
differential equation that characterizes minimal general rotational surfaces.
Also, they determined all pseudo-umbilical general rotational surfaces in
$\mathbb{E}^{4}$ \cite{Dursun2}. Furthermore, some special examples of general
rotational surfaces in $\mathbb{E}^{4}$ were given, the curvature properties
of these surfaces were investigated and the necessary and sufficient
conditions for general rotational surfaces to become pseudo-umbilical were
given in \cite{Arslan}. For further studies about (general) rotational
surfaces in $\mathbb{E}^{4},$ one can see \cite{Arslan2}, \cite{Dursun},
\cite{Ganchev2}, \cite{Vranceanu}, and etc. Also in \cite{Bencheva}, the
authors studied the timelike general rotational surfaces with non-null
(spacelike or timelike) plane meridian curves in the 4D Minkowski space.

On the other hand, the notion of the weighted manifold has started to be a
popular topic in different areas such as mathematics, physics, and economics,
recently. For instance, manifolds with density arise in physics when
considering surfaces or regions with differing physical density. An object may
have differing internal densities so in order to determine the object's mass
it is necessary to integrate volume weighted with density.

Manifold with density is a Riemannian manifold with positive density function
$\Phi=e^{\phi}$ used to weight volume and area. In terms of the underlying
Riemannian volume $dV_{0}$ and area $dA_{0}$, the new weighted volume and area
are given by%
\[
dV_{\phi}=\Phi dV_{0},\text{ }dA_{\phi}=\Phi dA_{0}.
\]
In this context, from the first variation of weighted area, Gromov firstly
introduced the notion of weighted mean curvature (or $\phi$-mean curvature) of
a hypersurface in an $n$-dimensional Riemannian manifold with density
$e^{\phi}$ as%
\[
H_{\phi}=H-\frac{1}{n-1}\frac{d\phi}{dN},
\]
where $H$ is the mean curvature and $N$ is the unit normal vector field of the
hypersurface \cite{Gromov}. A hypersurface is called weighted minimal (or
$\phi$-minimal) if its weighted mean curvature vanishes. Also, Corvin et al
introduced the notion of generalized weighted Gaussian curvature of a surface
on a manifold as%
\[
K_{\phi}=K-\Delta\phi,
\]
where $\Delta$ is the Laplacian operator and $K$ is the Gaussian curvature of
a surface \cite{Corwin}. A surface is called weighted flat (or $\phi$-flat) if
its weighted Gaussian curvature vanishes.

After these definitions, lots of studies were done by differential geometers
about curves and surfaces in different spaces with density. For instance Hieu
and Nam \cite{Hieu} classified the constant weighted curvature curves in the
plane with a log-linear density. Lopez \cite{lo} considered a log-linear
density $e^{\alpha x+\beta y+\gamma z}$ and he classified the $\phi$-minimal
translation surfaces and the $\phi$-minimal cyclic surfaces in $\mathbb{E}%
^{3}$. Also, in \cite{Yoon1} authors considered a 3D Euclidean space with
density $e^{-x^{2}-y^{2}}$ and they constructed all the helicoidal surfaces in
the space by solving the second-order non-linear ordinary differential
equation with the weighted Gaussian curvature and the weighted mean curvature
functions. Yoon and Y\"{u}zba\c{s}\i~\cite{Yoon2} completely classified affine
translation surfaces with zero weighted mean curvature. Rotational and ruled
surfaces generated by planar curves in $\mathbb{E}^{3}$ with density were
studied in \cite{Altin1} and \cite{Altin6}, respectively and Monge
hypersurfaces in $\mathbb{E}^{4}$ with density were studied in \cite{Altin3}.
In \cite{Altin5}, the authors gave important results for non-null curves with
constant weighted curvature in Lorentz-Minkowski plane with density and in
\cite{Kazanyy}, the authors studied generalized rotation surfaces in
$\mathbb{E}^{4}$ with density. For more details about manifolds with density
and some relative topics we refer to \cite{Altin4}, \cite{Belarbi},
\cite{Kazanyyyy}, \cite{Altin7}, \cite{lo2}, \cite{Morgan}-\cite{Nam},
\cite{Onder}, and etc. So, these studies motivated us to study rotational
surfaces in a 4D Minkowski space with density.

In the second section of this study, we recall the parametric expressions of
two types of timelike general rotational surfaces with non-null plane meridian
curves in $\mathbb{E}_{1}^{4}$. In Section 3, we obtain the weighted mean and
weighted Gaussian curvatures of the timelike general rotational surfaces in
$\mathbb{E}_{1}^{4}$ with density $e^{\lambda_{1}x^{2}+\lambda_{2}%
y^{2}+\lambda_{3}z^{2}+\lambda_{4}t^{2}},$ $\lambda_{i},$ $i\in\{1,2,3,4\},$
not all zero. Also, we give some results about weighted minimal and weighted
flat timelike general rotation surfaces in $\mathbb{E}_{1}^{4}$ with density
and construct some examples for these surfaces.

\section{\textbf{TIMELIKE GENERAL ROTATIONAL SURFACES WITH PLANE MERIDIAN
CURVES IN} $\mathbb{E}_{1}^{4}$}

In this section, we will recall the parametric expressions of two types of
timelike general rotational surfaces with plane meridian curves in
$\mathbb{E}_{1}^{4}$ and give some geometric invariants of these surfaces.

Let $\mathbb{E}_{1}^{4}$ be the 4D Minkowski space endowed with the metric
$\left\langle \text{ },\right\rangle $ of signature $(+,+,+,-)$ and
$Oe_{1}e_{2}e_{3}e_{4}$ be a fixed orthonormal coordinate system, i.e.
$e_{1}^{2}=e_{2}^{2}=e_{3}^{2}=-e_{4}^{2}=1$, giving the orientation of
$\mathbb{E}_{1}^{4}$. The standard flat metric is given in local coordinates
by $dx_{1}^{2}+dx_{2}^{2}+dx_{3}^{2}-dx_{4}^{2}$.

Let we consider a surface $\mathcal{S}_{1}$ in $\mathbb{E}_{1}^{4}$
parametrized by%
\begin{equation}
\mathcal{S}_{1}:X_{1}(u,v)=(f(u)\cos(\alpha v),f(u)\sin(\alpha v),g(u)\cosh
(\beta v),g(u)\sinh(\beta v)), \label{s1}%
\end{equation}
where $u\in J\subset%
\mathbb{R}
,$ $v\in\lbrack0,2\pi);$ $\alpha,$ $\beta$ are positive constants; $f(u)$ and
$g(u)$ are smooth functions, satisfying the inequalities%
\begin{equation}
\alpha^{2}f^{2}(u)-\beta^{2}g^{2}(u)<0,\text{ \ \ }f^{\prime2}(u)+g^{\prime
2}(u)>0. \label{s2}%
\end{equation}
Here, since the coefficients of the first fundamental form of $\mathcal{S}%
_{1}$ are
\[
E^{\mathcal{S}_{1}}=\text{\ }f^{\prime2}(u)+g^{\prime2}(u)>0,\text{
}F^{\mathcal{S}_{1}}=0,\text{ }G^{\mathcal{S}_{1}}=\alpha^{2}f^{2}%
(u)-\beta^{2}g^{2}(u)<0,
\]
the surface $\mathcal{S}_{1}$ is timelike and the meridian curve $c_{1}%
:x_{1}(u)=(f(u),0,g(u),0)$, $u\in J\subset%
\mathbb{R}
,$ is spacelike. By considering tangent frame fields of $\mathcal{S}_{1}$
defined by%
\begin{equation}
\left.
\begin{array}
[c]{l}%
t_{1}^{\mathcal{S}_{1}}=\frac{1}{\sqrt{f^{\prime2}+g^{\prime2}}}\left(
f^{\prime}\cos(\alpha v),f^{\prime}\sin(\alpha v),g^{\prime}\cosh(\beta
v),g^{\prime}\sinh(\beta v)\right)  ,\\
t_{2}^{\mathcal{S}_{1}}=\frac{1}{\sqrt{(\beta g)^{2}-(\alpha f)^{2}}}\left(
-\alpha f\sin(\alpha v),\alpha f\cos(\alpha v),\beta g\sinh(\beta v),\beta
g\cosh(\beta v)\right)  ,
\end{array}
\right\}  \label{s3}%
\end{equation}
we can choose the following normal frame fields of $\mathcal{S}_{1}:$%
\begin{equation}
\left.
\begin{array}
[c]{l}%
n_{1}^{\mathcal{S}_{1}}=\frac{1}{\sqrt{f^{\prime2}+g^{\prime2}}}\left(
g^{\prime}\cos(\alpha v),g^{\prime}\sin(\alpha v),-f^{\prime}\cosh(\beta
v),-f^{\prime}\sinh(\beta v)\right)  ,\\
n_{2}^{\mathcal{S}_{1}}=\frac{1}{\sqrt{(\beta g)^{2}-(\alpha f)^{2}}}\left(
-\beta g\sin(\alpha v),\beta g\cos(\alpha v),\alpha f\sinh(\beta v),\alpha
f\cosh(\beta v)\right)  .
\end{array}
\right\}  \label{s4}%
\end{equation}
Throughout this study, we state $f=f(u),$ $g=g(u),$ $f^{\prime}=\frac
{df(u)}{du},$ $g^{\prime}=\frac{dg(u)}{du},$ and so on.

Similarly, let we consider an another surface $\mathcal{S}_{2}$ in
$\mathbb{E}_{1}^{4}$ parametrized by%
\begin{equation}
\mathcal{S}_{2}:X_{2}(u,v)=(f(u)\cos(\alpha v),f(u)\sin(\alpha v),g(u)\sinh
(\beta v),g(u)\cosh(\beta v)), \label{s5}%
\end{equation}
where $u\in J\subset%
\mathbb{R}
,$ $v\in\lbrack0,2\pi);$ $\alpha,$ $\beta$ are positive constants; $f(u)$ and
$g(u)$ are smooth functions, satisfying the inequalities%
\begin{equation}
\text{\ }f^{\prime2}(u)-g^{\prime2}(u)<0,\text{ \ \ }\alpha^{2}f^{2}%
(u)+\beta^{2}g^{2}(u)>0. \label{s6}%
\end{equation}
Here, since the coefficients of the first fundamental form of $\mathcal{S}%
_{2}$ are
\[
E^{\mathcal{S}_{2}}=\text{\ }f^{\prime2}(u)-g^{\prime2}(u)<0,\text{
}F^{\mathcal{S}_{2}}=0,\text{ }G^{\mathcal{S}_{2}}=\alpha^{2}f^{2}%
(u)+\beta^{2}g^{2}(u)>0,
\]
the surface $\mathcal{S}_{2}$ is timelike and the meridian curve $c_{2}%
:x_{2}(u)=(f(u),0,0,g(u))$, $u\in J\subset%
\mathbb{R}
,$ is timelike. By considering tangent frame fields defined by%
\begin{equation}
\left.
\begin{array}
[c]{l}%
t_{1}^{\mathcal{S}_{2}}=\frac{1}{\sqrt{g^{\prime2}-f^{\prime2}}}\left(
f^{\prime}\cos(\alpha v),f^{\prime}\sin(\alpha v),g^{\prime}\sinh(\beta
v),g^{\prime}\cosh(\beta v)\right)  ,\\
t_{2}^{\mathcal{S}_{2}}=\frac{1}{\sqrt{(\alpha f)^{2}+(\beta g)^{2}}}\left(
-\alpha f\sin(\alpha v),\alpha f\cos(\alpha v),\beta g\cosh(\beta v),\beta
g\sinh(\beta v)\right)  ,
\end{array}
\right\}  \label{s7}%
\end{equation}
we can consider the following normal frame fields of $\mathcal{S}_{2}:$%
\begin{equation}
\left.
\begin{array}
[c]{l}%
n_{1}^{\mathcal{S}_{2}}=\frac{1}{\sqrt{(\alpha f)^{2}+(\beta g)^{2}}}\left(
\beta g\sin(\alpha v),-\beta g\cos(\alpha v),\alpha f\cosh(\beta v),\alpha
f\sinh(\beta v)\right)  ,\\
n_{2}^{\mathcal{S}_{2}}=\frac{1}{\sqrt{-f^{\prime2}+g^{\prime2}}}\left(
g^{\prime}\cos(\alpha v),g^{\prime}\sin(\alpha v),f^{\prime}\sinh(\beta
v),f^{\prime}\cosh(\beta v)\right)  .
\end{array}
\right\}  \label{s8}%
\end{equation}
We call the surfaces $\mathcal{S}_{1}$ and $\mathcal{S}_{2}$ as
\textit{general rotational surface of first type }and\textit{ general
rotational surface of second type}, respectively.

Also, the mean curvature vectors $\overrightarrow{H^{\mathcal{S}_{i}}}$ and
Gaussian curvatures $K^{\mathcal{S}_{i}}$ of these two types of timelike
general rotational surfaces $\mathcal{S}_{i}$ parametrized by (\ref{s1}) and
(\ref{s5}) are given by%
\begin{equation}
\overrightarrow{H^{\mathcal{S}_{i}}}=\frac{\left(  f^{\prime}g^{\prime\prime
}-f^{\prime\prime}g^{\prime}\right)  \left(  (\alpha f)^{2}+(-1)^{i}(\beta
g)^{2}\right)  -(-1)^{i}\left(  \alpha^{2}fg^{\prime}+\beta^{2}gf^{\prime
}\right)  \left(  -(-1)^{i}f^{\prime2}+g^{\prime2}\right)  }{2\left(
-(-1)^{i}f^{\prime2}+g^{\prime2}\right)  ^{3/2}\left(  (\alpha f)^{2}%
+(-1)^{i}(\beta g)^{2}\right)  }n_{i}^{\mathcal{S}_{i}}, \label{s9}%
\end{equation}
where
\begin{equation}
{\small {H_{j}^{\mathcal{S}_{i}}=-\frac{\left(  (-1)^{i}+(-1)^{j}\right)
\left(  \left(  \alpha^{2}fg^{\prime}+\beta^{2}gf^{\prime}\right)  \left(
-(-1)^{i}f^{\prime2}+g^{\prime2}\right)  -(-1)^{i}\left(  (\alpha
f)^{2}+(-1)^{i}(\beta g)^{2}\right)  \left(  f^{\prime}g^{\prime\prime
}-f^{\prime\prime}g^{\prime}\right)  \right)  }{4\left(  (\alpha
f)^{2}+(-1)^{i}(\beta g)^{2}\right)  \left(  -(-1)^{i}f^{\prime2}+g^{\prime
2}\right)  ^{3/2}}}} \label{s10}%
\end{equation}
and%
\begin{equation}
K^{\mathcal{S}_{i}}=\frac{\alpha^{2}\beta^{2}\left(  gf^{\prime}-fg^{\prime
}\right)  ^{2}\left(  -(-1)^{i}f^{\prime2}+g^{\prime2}\right)  +\left(
(-1)^{i}(\alpha f)^{2}+(\beta g)^{2}\right)  \left(  \alpha^{2}fg^{\prime
}+\beta^{2}gf^{\prime}\right)  \left(  f^{\prime\prime}g^{\prime}-f^{\prime
}g^{\prime\prime}\right)  }{\left(  (\alpha f)^{2}+(-1)^{i}(\beta
g)^{2}\right)  ^{2}\left(  -(-1)^{i}f^{\prime2}+g^{\prime2}\right)  ^{2}},
\label{s11}%
\end{equation}
respectively. Here, $H_{j}^{\mathcal{S}_{i}}$ are the $j$-th mean curvatures
of $\mathcal{S}_{i}$. For more details about construction of the timelike
general rotational surfaces with non-null plane meridian curves in
$\mathbb{E}_{1}^{4}$, we refer to \cite{Bencheva}.

Also we must note that, from now on, we will assume that $i,j\in\{1,2\}$ and
we will state $\varepsilon=(-1)^{i}$ throughout this study.

\section{\textbf{TIMELIKE\ GENERAL ROTATIONAL SURFACES IN }$\mathbb{E}_{1}%
^{4}$\textbf{ WITH DENSITY }$e^{\lambda_{1}x^{2}+\lambda_{2}y^{2}+\lambda
_{3}z^{2}+\lambda_{4}t^{2}}$}

In this section, we obtain the weighted mean and weighted Gaussian curvatures
of the timelike general rotational surfaces $\mathcal{S}_{i}$ parametrized by
(\ref{s1}) and (\ref{s5}) in $\mathbb{E}_{1}^{4}$ with density $e^{\lambda
_{1}x^{2}+\lambda_{2}y^{2}+\lambda_{3}z^{2}+\lambda_{4}t^{2}},$ where
$\lambda_{i}$ ($i=1,2,3,4$) are not all zero. Also, we give some results about
weighted minimal and weighted flat timelike general rotational surfaces in
$\mathbb{E}_{1}^{4}$ with density.

Firstly, we define the weighted mean curvature vector field of a surface
$\mathcal{S}$ in $\mathbb{E}_{1}^{4}$ as%
\begin{equation}
\overrightarrow{H_{\phi}^{\mathcal{S}}}=(H_{1}^{\mathcal{S}})_{\phi}%
n_{1}^{\mathcal{S}}+(H_{2}^{\mathcal{S}})_{\phi}n_{2}^{\mathcal{S}},
\label{s12}%
\end{equation}
where $H_{j}^{\mathcal{S}}$ are the $j$-th mean curvatures of $\mathcal{S}$;
$n_{j}^{\mathcal{S}}$ are normal vectors of $\mathcal{S}$ and
\begin{equation}
(H_{j}^{\mathcal{S}})_{\phi}=H_{j}^{\mathcal{S}}-\frac{1}{2}\frac{d\phi
}{dn_{j}^{\mathcal{S}}},\text{ }j\in\{1,2\} \label{s13}%
\end{equation}
are the $j$-th weighted mean curvature functions of $\mathcal{S}$. Also, the
weighted mean curvature function of $\mathcal{S}$ is $H_{\phi}^{\mathcal{S}%
}=\left\Vert \overrightarrow{H_{\phi}^{\mathcal{S}}}\right\Vert .$

Thus from (\ref{s4}), (\ref{s8}) and (\ref{s13}), we obtain the first and
second weighted mean curvature functions of the timelike general rotational
surfaces $\mathcal{S}_{i}$ parametrized by (\ref{s1}) and (\ref{s5}) in
$\mathbb{E}_{1}^{4}$ with density $e^{\lambda_{1}x^{2}+\lambda_{2}%
y^{2}+\lambda_{3}z^{2}+\lambda_{4}t^{2}}$ as follows:%

\begin{equation}
\left(  H_{j}^{\mathcal{S}_{i}}\right)  _{\phi}=\left\{
\begin{array}
[c]{l}%
-\frac{fg\left(  \varepsilon\beta(\lambda_{1}-\lambda_{2})\sin(2\alpha
v)+\alpha(\lambda_{3}+\lambda_{4})\sinh(2\beta v)\right)  }{2\sqrt
{\varepsilon(\alpha f)^{2}+(\beta g)^{2}}},\text{\ if }i\neq j\\
\\
\frac{\left(
\begin{array}
[c]{l}%
\left(
\begin{array}
[c]{l}%
\left(  \varepsilon(\lambda_{4}-\lambda_{3})+(\lambda_{3}+\lambda_{4}%
)\cosh(2\beta v)\right)  \left(  f^{\prime2}-\varepsilon g^{\prime2}\right)
{\small gf}^{\prime}\\
{\small +\varepsilon}\left(  \lambda_{1}+\lambda_{2}+(\lambda_{1}-\lambda
_{2})\cos(2\alpha v)\right)  \left(  f^{\prime2}-\varepsilon g^{\prime
2}\right)  {\small fg}^{\prime}\\
{\small -g}^{\prime}{\small f}^{\prime\prime}{\small +f}^{\prime}%
{\small g}^{\prime\prime}%
\end{array}
\right)  \left(  (\alpha f)^{2}+\varepsilon(\beta g)^{2}\right) \\
{\small +}\left(  \alpha^{2}fg^{\prime}+\beta^{2}gf^{\prime}\right)  \left(
f^{\prime2}-\varepsilon g^{\prime2}\right)
\end{array}
\right)  }{2\left(  (\alpha f)^{2}+\varepsilon(\beta g)^{2}\right)  \left(
-\varepsilon f^{\prime2}+g^{\prime2}\right)  ^{3/2}}{\small ,\text{\ if }i=j}%
\end{array}
\right.  \label{s14}%
\end{equation}

Hence, using (\ref{s4}), (\ref{s8}) and (\ref{s14}) in (\ref{s12}), from
$H_{\phi}^{\mathcal{S}}=\left\Vert \overrightarrow{H_{\phi}^{\mathcal{S}}%
}\right\Vert $ we get the following proposition:

\begin{proposition}
The weighted mean curvatures of the timelike general rotational surfaces
\textit{of first and second types }$\mathcal{S}_{i}$ parametrized by
(\ref{s1}) and (\ref{s5}) in $\mathbb{E}_{1}^{4}$ with density $e^{\lambda
_{1}x^{2}+\lambda_{2}y^{2}+\lambda_{3}z^{2}+\lambda_{4}t^{2}}$ are%
\begin{equation}
H_{\phi}^{\mathcal{S}_{i}}={\small \sqrt{\frac{\left(
\begin{array}
[c]{l}%
\left(
\begin{array}
[c]{l}%
\left(  (\alpha f)^{2}+\varepsilon(\beta g)^{2}\right)  \left(  \varepsilon
(\lambda_{1}+\lambda_{2}+(\lambda_{1}-\lambda_{2})\cos(2\alpha v))fg^{\prime
}\left(  f^{\prime2}-\varepsilon g^{\prime2}\right)  -g^{\prime}%
f^{\prime\prime}+f^{\prime}g^{\prime\prime}\right) \\
+\left(  \left(  \varepsilon(\lambda_{4}-\lambda_{3})+(\lambda_{3}+\lambda
_{4})\cosh(2\beta v)\right)  \left(  (\alpha f)^{2}+\varepsilon(\beta
g)^{2}\right)  f^{\prime}g+\alpha^{2}fg^{\prime}+\beta^{2}gf^{\prime}\right)
\left(  f^{\prime2}-\varepsilon g^{\prime2}\right)
\end{array}
\right)  ^{2}\\
+f^{2}g^{2}\left(  \varepsilon(\alpha f)^{2}+(\beta g)^{2}\right)  \left(
\varepsilon\beta(\lambda_{1}-\lambda_{2})\sin(2\alpha v)+\alpha(\lambda
_{3}+\lambda_{4})\sinh(2\beta v)\right)  ^{2}\left(  -\varepsilon f^{\prime
2}+g^{\prime2}\right)  ^{3}%
\end{array}
\right)  }{4\left(  (\alpha f)^{2}+\varepsilon(\beta g)^{2}\right)
^{2}\left(  -\varepsilon f^{\prime2}+g^{\prime2}\right)  ^{3}}}.} \label{s15}%
\end{equation}

\end{proposition}

\subsection{Weighted Minimal Timelike General Rotational surfaces in
$\mathbb{E}_{1}^{4}$ with Density $e^{\lambda(x^{2}+y^{2})+\mu(z^{2}-t^{2})}$}

\ 

Let we take $\lambda_{1}=\lambda_{2}=\lambda$ and $\lambda_{3}=-\lambda
_{4}=\mu$ in (\ref{s15}), i.e. let the density be $e^{\lambda(x^{2}+y^{2}%
)+\mu(z^{2}-t^{2})},$ where $\lambda$ and $\mu$ are not all zero. Then the
weighted mean curvatures of the timelike general rotational surfaces of first
and second types $\mathcal{S}_{i}$ parametrized by (\ref{s1}) and (\ref{s5})
in $\mathbb{E}_{1}^{4}$ with density $e^{\lambda(x^{2}+y^{2})+\mu(z^{2}%
-t^{2})}$ are obtained as%
\begin{equation}
H_{\phi}^{\mathcal{S}_{i}}=\frac{\left\vert
\begin{array}
[c]{l}%
\left(  \left(  \alpha^{2}fg^{\prime}+\beta^{2}gf^{\prime}\right)
-2\varepsilon\mu\left(  (\alpha f)^{2}+\varepsilon(\beta g)^{2}\right)
f^{\prime}g\right)  \left(  f^{\prime2}-\varepsilon g^{\prime2}\right) \\
+\left(  (\alpha f)^{2}+\varepsilon(\beta g)^{2}\right)  \left(
2\varepsilon\lambda\left(  f^{\prime2}-\varepsilon g^{\prime2}\right)
fg^{\prime}-g^{\prime}f^{\prime\prime}+f^{\prime}g^{\prime\prime}\right)
\end{array}
\right\vert }{2\left(  \varepsilon(\alpha f)^{2}+(\beta g)^{2}\right)  \left(
-\varepsilon f^{\prime2}+g^{\prime2}\right)  ^{3/2}}. \label{ss16}%
\end{equation}
Thus, we have

\begin{theorem}
The timelike general rotational surfaces of first and second types
$\mathcal{S}_{i}$ parametrized by (\ref{s1}) and (\ref{s5}) in $\mathbb{E}%
_{1}^{4}$ with density $e^{\lambda(x^{2}+y^{2})+\mu(z^{2}-t^{2})}$ are
weighted minimal if and only if the equation%
\begin{align}
&  \left(  \left(  -2\varepsilon\mu\alpha^{2}f^{2}+\beta^{2}\right)
f^{\prime}g+\alpha^{2}fg^{\prime}-2\beta^{2}\mu g^{3}f^{\prime}\right)
\left(  f^{\prime2}-\varepsilon g^{\prime2}\right) \nonumber\\
&  +\left(  (\alpha f)^{2}+\varepsilon(\beta g)^{2}\right)  \left(  \left(
-2\lambda f\left(  -\varepsilon f^{\prime2}+g^{\prime2}\right)  -f^{\prime
\prime}\right)  g^{\prime}+f^{\prime}g^{\prime\prime}\right)  =0 \label{s17}%
\end{align}
holds.
\end{theorem}

Now, let we construct two examples for the weighted minimal timelike general
rotational surfaces of first and second types $\mathcal{S}_{i}.$

\begin{example}
If we take $f(u)=\sin u,$ $g(u)=\cos u$ and $\alpha=\beta=1$ in (\ref{s1}),
then we get the following timelike general rotational surface of first type:%
\begin{equation}
\mathcal{S}_{1}:X(u,v)=(\sin u\cos v,\sin u\sin v,\cos u\cosh v,\cos u\sinh
v), \label{ex1}%
\end{equation}
where we suppose $u\in\left(  -\frac{\pi}{4},\frac{\pi}{4}\right)  .$ The
weighted mean curvature of this surface in $\mathbb{E}_{1}^{4}$ with density
$e^{\lambda_{1}x^{2}+\lambda_{2}y^{2}+\lambda_{3}z^{2}+\lambda_{4}t^{2}}$ is%
\begin{equation}
H_{\phi}^{\mathcal{S}_{1}}=\frac{1}{2}\sqrt{\left(
\begin{array}
[c]{l}%
\left(
\begin{array}
[c]{l}%
(\lambda_{1}+\lambda_{2}+(\lambda_{1}-\lambda_{2})\cos(2v))\sin^{2}u-2\\
+(\lambda_{3}-\lambda_{4}+\left(  \lambda_{3}+\lambda_{4}\right)
\cosh(2v))\cos^{2}u
\end{array}
\right)  ^{2}\cos^{2}(2u)\\
+\frac{1}{8}\sin(2u)\sin(4u)\left(  \left(  \lambda_{1}-\lambda_{2}\right)
\sin(2v)-\left(  \lambda_{3}+\lambda_{4}\right)  \sinh(2v)\right)  ^{2}%
\end{array}
\right)  \sec^{2}(2u)} \label{ex2}%
\end{equation}
Taking $\lambda_{1}=\lambda_{2}=\lambda$ and $\lambda_{3}=-\lambda_{4}=\mu,$
we obtain the weighted mean curvature of the surface (\ref{ex1}) in
$\mathbb{E}_{1}^{4}$ with density $e^{\lambda(x^{2}+y^{2})+\mu(z^{2}-t^{2})}$
as%
\begin{equation}
H_{\phi}^{\mathcal{S}_{1}}=\left\vert -1+\lambda\sin^{2}u+\mu\cos
^{2}u\right\vert . \label{ex3}%
\end{equation}
Thus, the timelike general rotational surface of first type $\mathcal{S}_{1}$
parametrized by (\ref{ex1}) is weighted minimal in $\mathbb{E}_{1}^{4}$ with
density $e^{x^{2}+y^{2}+z^{2}-t^{2}}.$ In Figure 1(a), one can see the
projection of the surface (\ref{ex1}) into $x_{1}x_{2}x_{3}$-space.
\end{example}

\begin{example}
If we take $f(u)=\cosh u,$ $g(u)=\sinh u$ and $\alpha=3,$ $\beta=2$ in
(\ref{s5}), then we get the following timelike general rotational surface of
second type:%
\begin{equation}
\mathcal{S}_{2}:X(u,v)=(\cosh u\cos(3v),\cosh u\sin(3v),\sinh u\sinh(2v),\sinh
u\cosh(2v)). \label{ex4}%
\end{equation}
The weighted mean curvature of this surface in $\mathbb{E}_{1}^{4}$ with
density $e^{\lambda_{1}x^{2}+\lambda_{2}y^{2}+\lambda_{3}z^{2}+\lambda
_{4}t^{2}}$ is%
\begin{equation}
H_{\phi}^{\mathcal{S}_{2}}=\frac{1}{2}\sqrt{%
\begin{array}
[c]{l}%
\left(  2+\left(  \lambda_{1}+\lambda_{2}+\left(  \lambda_{1}-\lambda
_{2}\right)  \cos(6v)\right)  \cosh^{2}u+\left(  -\lambda_{3}+\lambda
_{4}+\left(  \lambda_{3}+\lambda_{4}\right)  \cosh(4v)\sinh^{2}u\right)
\right)  ^{2}\\
+\frac{\left(  2\left(  \lambda_{1}-\lambda_{2})\sin(6v)\right)  +3\left(
\lambda_{3}+\lambda_{4}\right)  \sinh(4v)\right)  ^{2}\cosh^{2}u\sinh^{2}%
u}{9\cosh^{2}u+4\sinh^{2}u}%
\end{array}
} \label{ex5}%
\end{equation}
Taking $\lambda_{1}=\lambda_{2}=\lambda$ and $\lambda_{3}=-\lambda_{4}=\mu,$
we obtain the weighted mean curvature of the surface (\ref{ex4}) in
$\mathbb{E}_{1}^{4}$ with density $e^{\lambda(x^{2}+y^{2})+\mu(z^{2}-t^{2})}$
as%
\begin{equation}
H_{\phi}^{\mathcal{S}_{2}}=\left\vert 1+\lambda\cosh^{2}u-\mu\sinh
^{2}u\right\vert . \label{ex6}%
\end{equation}
Thus, the timelike general rotational surface of second type $\mathcal{S}_{2}$
parametrized by (\ref{ex4}) is weighted minimal in $\mathbb{E}_{1}^{4}$ with
density $e^{-x^{2}-y^{2}-z^{2}+t^{2}}.$ In Figure 1(b), one can see the
projection of the surface (\ref{ex4}) into $x_{1}x_{2}x_{3}$-space.
\end{example}

\begin{figure}[H]
\centering
\includegraphics[
	height=3.1in, width=5.4in
	]	{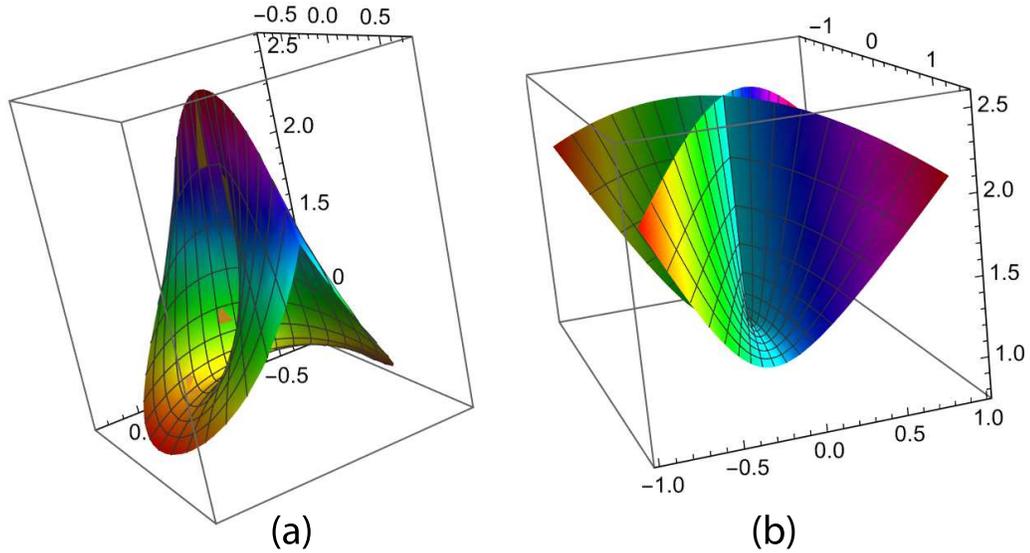}\caption{The projection of the weighted minimal timelike
general rotational surface of first type (\ref{ex1}) into $x_{1}x_{2}x_{3}%
$-space \textbf{(a)} and the projection of the weighted minimal timelike
general rotational surface of second type (\ref{ex4}) into $x_{1}x_{2}x_{3}%
$-space \textbf{(b)}}%
\label{fig:1}%
\end{figure}

Now, let us suppose that $f(u)=k,$ where $k$ is a nonzero constant. In this
case, from (\ref{ss16}), the weighted mean curvature is%
\begin{equation}
H_{\phi}^{\mathcal{S}_{i}}=\varepsilon\frac{\left\vert k\left(  \alpha
^{2}\left(  1+2\varepsilon k^{2}\lambda\right)  +2\beta^{2}\lambda
g^{2}\right)  \right\vert }{2\left(  \alpha^{2}k^{2}+\varepsilon(\beta
g)^{2}\right)  }. \label{s18}%
\end{equation}
If we find the function $g(u)$ from (\ref{s18}) according to the weighted mean
curvature $\left(  H_{\phi}\right)  ^{\mathcal{S}_{i}}$, then we get%
\begin{equation}
g(u)=\left\{
\begin{array}
[c]{l}%
\mp\frac{\alpha}{\beta}\frac{\sqrt{k\left(  \varepsilon-2H_{\phi}%
^{\mathcal{S}_{i}}k+2k^{2}\lambda\right)  }}{\sqrt{2\varepsilon\left(
H_{\phi}^{\mathcal{S}_{i}}-k\lambda\right)  }},\text{\ if }k\left(  \alpha
^{2}\left(  1+2\varepsilon k^{2}\lambda\right)  +2\beta^{2}\lambda
g^{2}\right)  >0\\
\\
\mp\frac{\alpha}{\beta}\frac{\sqrt{k\left(  \varepsilon+2H_{\phi}%
^{\mathcal{S}_{i}}k+2k^{2}\lambda\right)  }}{\sqrt{-2\varepsilon\left(
H_{\phi}^{\mathcal{S}_{i}}+k\lambda\right)  }},\text{\ if }k\left(  \alpha
^{2}\left(  1+2\varepsilon k^{2}\lambda\right)  +2\beta^{2}\lambda
g^{2}\right)  <0.
\end{array}
\right.  \label{s19}%
\end{equation}
Thus, we can state

\begin{theorem}
The timelike general rotational surfaces of first and second types
$\mathcal{S}_{i}$ in $\mathbb{E}_{1}^{4}$ with density $e^{\lambda(x^{2}%
+y^{2})+\mu(z^{2}-t^{2})}$ can be parametrized in terms of the weighted mean
curvature by the following:

For $k\left(  \alpha^{2}\left(  1+2\varepsilon k^{2}\lambda\right)
+2\beta^{2}\lambda g^{2}\right)  >0$%
\begin{align}
\mathcal{S}_{1}  &  :X_{1}(u,v)=(k\cos(\alpha v),k\sin(\alpha v),\mp
\frac{\alpha}{\beta}\frac{\sqrt{k\left(  -1-2H_{\phi}^{\mathcal{S}_{1}%
}k+2k^{2}\lambda\right)  }}{\sqrt{-2\left(  H_{\phi}^{\mathcal{S}_{1}%
}-k\lambda\right)  }}\cosh(\beta v),\nonumber\\
&  \text{
\ \ \ \ \ \ \ \ \ \ \ \ \ \ \ \ \ \ \ \ \ \ \ \ \ \ \ \ \ \ \ \ \ \ \ \ \ \ \ \ }%
\mp\frac{\alpha}{\beta}\frac{\sqrt{k\left(  -1-2H_{\phi}^{\mathcal{S}_{1}%
}k+2k^{2}\lambda\right)  }}{\sqrt{-2\left(  H_{\phi}^{\mathcal{S}_{1}%
}-k\lambda\right)  }}\sinh(\beta v)),\label{s21}\\
\mathcal{S}_{2}  &  :X_{2}(u,v)=(k\cos(\alpha v),k\sin(\alpha v),\mp
\frac{\alpha}{\beta}\frac{\sqrt{k\left(  1-2H_{\phi}^{\mathcal{S}_{2}}%
k+2k^{2}\lambda\right)  }}{\sqrt{2\left(  H_{\phi}^{\mathcal{S}_{2}}%
-k\lambda\right)  }}\sinh(\beta v),\nonumber\\
&  \text{
\ \ \ \ \ \ \ \ \ \ \ \ \ \ \ \ \ \ \ \ \ \ \ \ \ \ \ \ \ \ \ \ \ \ \ \ \ \ \ \ }%
\mp\frac{\alpha}{\beta}\frac{\sqrt{k\left(  1-2H_{\phi}^{\mathcal{S}_{2}%
}k+2k^{2}\lambda\right)  }}{\sqrt{2\left(  H_{\phi}^{\mathcal{S}_{2}}%
-k\lambda\right)  }}\cosh(\beta v)). \label{s22}%
\end{align}
For $k\left(  \alpha^{2}\left(  1+2\varepsilon k^{2}\lambda\right)
+2\beta^{2}\lambda g^{2}\right)  <0$%
\begin{align}
\mathcal{S}_{1}  &  :X_{1}(u,v)=(k\cos(\alpha v),k\sin(\alpha v),\mp
\frac{\alpha}{\beta}\frac{\sqrt{k\left(  -1+2H_{\phi}^{\mathcal{S}_{1}%
}k+2k^{2}\lambda\right)  }}{\sqrt{2\left(  H_{\phi}^{\mathcal{S}_{1}}%
+k\lambda\right)  }}\cosh(\beta v),\nonumber\\
&  \text{
\ \ \ \ \ \ \ \ \ \ \ \ \ \ \ \ \ \ \ \ \ \ \ \ \ \ \ \ \ \ \ \ \ \ \ \ \ \ \ \ }%
\mp\frac{\alpha}{\beta}\frac{\sqrt{k\left(  -1+2H_{\phi}^{\mathcal{S}_{1}%
}k+2k^{2}\lambda\right)  }}{\sqrt{2\left(  H_{\phi}^{\mathcal{S}_{1}}%
+k\lambda\right)  }}\sinh(\beta v)),\label{s23}\\
\mathcal{S}_{2}  &  :X_{2}(u,v)=(k\cos(\alpha v),k\sin(\alpha v),\mp
\frac{\alpha}{\beta}\frac{\sqrt{k\left(  1+2H_{\phi}^{\mathcal{S}_{2}}%
k+2k^{2}\lambda\right)  }}{\sqrt{-2\left(  H_{\phi}^{\mathcal{S}_{2}}%
+k\lambda\right)  }}\sinh(\beta v),\nonumber\\
&  \text{
\ \ \ \ \ \ \ \ \ \ \ \ \ \ \ \ \ \ \ \ \ \ \ \ \ \ \ \ \ \ \ \ \ \ \ \ \ \ \ \ }%
\mp\frac{\alpha}{\beta}\frac{\sqrt{k\left(  1+2H_{\phi}^{\mathcal{S}_{2}%
}k+2k^{2}\lambda\right)  }}{\sqrt{-2\left(  H_{\phi}^{\mathcal{S}_{2}%
}+k\lambda\right)  }}\cosh(\beta v)). \label{s24}%
\end{align}
Here $\alpha,$ $\beta,$ $k$ and $\lambda$ are nonzero constants ($\alpha>0,$
$\beta>0$).
\end{theorem}

If we assume that $H_{\phi}^{\mathcal{S}_{i}}$ is constant in (\ref{s19}), we
have $g$ is a constant, too. Hence,

\begin{corollary}
The weighted mean curvatures of the timelike general rotational surfaces of
first and second types $\mathcal{S}_{i}$ in $\mathbb{E}_{1}^{4}$ with density
$e^{\lambda(x^{2}+y^{2})+\mu(z^{2}-t^{2})}$ cannot be constant when $f(u)$ is
a nonzero constant.
\end{corollary}

Now, we will suppose that $g(u)=k,$ where $k$ is a nonzero constant for the
timelike general rotational surfaces of first type $\mathcal{S}_{1}$ in
$\mathbb{E}_{1}^{4}$ with density $e^{\lambda(x^{2}+y^{2})+\mu(z^{2}-t^{2})}$.
Here we must note that, from (\ref{s6}), $g(u)$ cannot be constant for the
timelike general rotational surfaces of second type $\mathcal{S}_{2}$ in
$\mathbb{E}_{1}^{4}$.

With similar procedure, if we suppose that $g(u)=l,$ $l$ is a nonzero
constant, then the from (\ref{ss16}), the weighted mean curvature in this case
for $\mathcal{S}_{1}$ is%
\begin{equation}
H_{\phi}^{\mathcal{S}_{1}}=\frac{\left\vert l\left(  2\alpha^{2}\mu
f^{2}+\beta^{2}\left(  1-2l^{2}\mu\right)  \right)  \right\vert }{2\left(
-(\alpha f)^{2}+(\beta l)^{2}\right)  }. \label{s25}%
\end{equation}
If we find the function $f(u)$ from (\ref{s25}) according to the weighted mean
curvature $H_{\phi}^{\mathcal{S}_{1}}$, then we get%
\begin{equation}
f(u)=\left\{
\begin{array}
[c]{l}%
\mp\frac{\beta}{\alpha}\frac{\sqrt{l\left(  -1+2H_{\phi}^{\mathcal{S}_{1}%
}l+2l^{2}\mu\right)  }}{\sqrt{2\left(  H_{\phi}^{\mathcal{S}_{1}}+l\mu\right)
}},\text{\ if }l\left(  \beta^{2}\left(  1-2l^{2}\mu\right)  +2\alpha^{2}\mu
f^{2}\right)  >0\\
\\
\mp\frac{\beta}{\alpha}\frac{\sqrt{l\left(  -1-2H_{\phi}^{\mathcal{S}_{1}%
}l+2l^{2}\mu\right)  }}{\sqrt{2\left(  -H_{\phi}^{\mathcal{S}_{1}}%
+l\mu\right)  }},\text{\ if }l\left(  \beta^{2}\left(  1-2l^{2}\mu\right)
+2\alpha^{2}\mu f^{2}\right)  <0.
\end{array}
\right.  \label{s26}%
\end{equation}
Hence, we have

\begin{theorem}
The timelike general rotational surface of first type $\mathcal{S}_{1}$ in
$\mathbb{E}_{1}^{4}$ with density $e^{\lambda(x^{2}+y^{2})+\mu(z^{2}-t^{2})}$
can be parametrized in terms of the weighted mean curvature by the following:

For $l\left(  \beta^{2}\left(  1-2l^{2}\mu\right)  +2\alpha^{2}\mu
f^{2}\right)  >0$%
\begin{align}
\mathcal{S}_{1}  &  :X_{1}(u,v)=(\mp\frac{\beta}{\alpha}\frac{\sqrt{l\left(
-1+2H_{\phi}^{\mathcal{S}_{1}}l+2l^{2}\mu\right)  }}{\sqrt{2\left(  H_{\phi
}^{\mathcal{S}_{1}}+l\mu\right)  }}\cos(\alpha v),\mp\frac{\beta}{\alpha}%
\frac{\sqrt{l\left(  -1+2H_{\phi}^{\mathcal{S}_{1}}l+2l^{2}\mu\right)  }%
}{\sqrt{2\left(  H_{\phi}^{\mathcal{S}_{1}}+l\mu\right)  }}\sin(\alpha
v),\nonumber\\
&  \text{ \ \ \ \ \ \ \ \ \ \ \ \ \ \ \ \ }l\cosh(\beta v),l\sinh(\beta v))
\label{s28}%
\end{align}
and for $l\left(  \beta^{2}\left(  1-2l^{2}\mu\right)  +2\alpha^{2}\mu
f^{2}\right)  <0$
\begin{align}
\mathcal{S}_{1}  &  :X_{1}(u,v)=(\mp\frac{\beta}{\alpha}\frac{\sqrt{l\left(
-1-2H_{\phi}^{\mathcal{S}_{1}}l+2l^{2}\mu\right)  }}{\sqrt{2\left(  -H_{\phi
}^{\mathcal{S}_{1}}+l\mu\right)  }}\cos(\alpha v),\mp\frac{\beta}{\alpha}%
\frac{\sqrt{l\left(  -1-2H_{\phi}^{\mathcal{S}_{1}}l+2l^{2}\mu\right)  }%
}{\sqrt{2\left(  -H_{\phi}^{\mathcal{S}_{1}}+l\mu\right)  }}\sin(\alpha
v),\nonumber\\
&  \text{ \ \ \ \ \ \ \ \ \ \ \ \ \ \ \ \ }l\cosh(\beta v),l\sinh(\beta v)),
\label{s29}%
\end{align}
where $\alpha,$ $\beta,$ $l$ and $\mu$ are nonzero constants ($\alpha>0,$
$\beta>0$).
\end{theorem}

Thus, if we assume that $H_{\phi}^{\mathcal{S}_{1}}$ is constant in
(\ref{s26}), we have $f$ is a constant, too. Therefore, we have

\begin{corollary}
The weighted mean curvature of the timelike general rotational surface of
first type $\mathcal{S}_{1}$ in $\mathbb{E}_{1}^{4}$ with density
$e^{\lambda(x^{2}+y^{2})+\mu(z^{2}-t^{2})}$ cannot be constant when $g(u)$ is
a nonzero constant.
\end{corollary}

Now, let us suppose that $f(u)=u$. In this case, from (\ref{ss16}), the
weighted mean curvature is%
\begin{equation}
H_{\phi}^{\mathcal{S}_{i}}=\frac{\left\vert
\begin{array}
[c]{l}%
\left(  \beta^{2}g+u\alpha^{2}g^{\prime}-2\varepsilon\mu g\left(  (\alpha
u)^{2}+\varepsilon(\beta g)^{2}\right)  \right)  \left(  1-\varepsilon
g^{\prime2}\right) \\
+\left(  (\alpha u)^{2}+\varepsilon(\beta g)^{2}\right)  \left(
2\varepsilon\lambda ug^{\prime}\left(  1-\varepsilon g^{\prime2}\right)
+g^{\prime\prime}\right)
\end{array}
\right\vert }{2\left(  \varepsilon(\alpha u)^{2}+(\beta g)^{2}\right)  \left(
g^{\prime2}-\varepsilon\right)  ^{3/2}}. \label{s30}%
\end{equation}

\begin{theorem}
\label{teo1}The timelike general rotational surfaces of first and second types
$\mathcal{S}_{i}$ in $\mathbb{E}_{1}^{4}$ with density $e^{\lambda(x^{2}%
+y^{2})+\mu(z^{2}-t^{2})}$ are weighted minimal if and only if, up to
parametrization, the meridian curves are determined by $c_{1}:x_{1}%
(u)=(u,0,g(u),0)$ and $c_{2}:x_{2}(u)=(u,0,0,g(u)),$ where $g(u)$ is a
solution to the following differential equations:

For $\mathcal{S}_{1}$
\begin{equation}
\left(  \arctan\left(  g^{\prime}(u)\right)  \right)  ^{\prime}=A_{1}
\label{s31}%
\end{equation}
and for $\mathcal{S}_{2}$
\begin{equation}
\left(  \log\left(  \frac{1-g^{\prime}(u)}{1+g^{\prime}(u)}\right)  \right)
^{\prime}=2A_{1}, \label{s32}%
\end{equation}
where
\begin{equation}
A_{1}=\frac{\left(  2(\lambda ug^{\prime}(u)-\mu g(u))\right)  \left(  (\alpha
u)^{2}+\varepsilon(\beta g(u))^{2}\right)  +\varepsilon\left(  \beta
^{2}g(u)+\alpha^{2}ug^{\prime}(u)\right)  }{(\alpha u)^{2}+\varepsilon(\beta
g(u))^{2}}. \label{s33}%
\end{equation}

\end{theorem}

\begin{proof}
If the timelike general rotational surfaces of first and second types
$\mathcal{S}_{i}$ in $\mathbb{E}_{1}^{4}$ with density $e^{\lambda(x^{2}%
+y^{2})+\mu(z^{2}-t^{2})}$ are weighted minimal, then from (\ref{s30}) we have%
\begin{equation}
\frac{g^{\prime\prime}(u)}{-\varepsilon+g^{\prime}(u)^{2}}=\frac{\left(
2\left(  \lambda ug^{\prime}(u)-\mu g(u)\right)  \right)  \left(  (\alpha
u)^{2}+\varepsilon(\beta g(u))^{2}\right)  +\varepsilon\left(  \beta
^{2}g(u)+\alpha^{2}ug^{\prime}(u)\right)  }{(\alpha u)^{2}+\varepsilon(\beta
g(u))^{2}}. \label{s33'}%
\end{equation}
Thus from (\ref{s33'}), we get (\ref{s31}) and (\ref{s32}) for $\varepsilon
=-1$ and $\varepsilon=1,$ respectively.
\end{proof}

Now, let us suppose that $g(u)=u$. In this case, from (\ref{ss16}), the
weighted mean curvature is%
\begin{equation}
H_{\phi}^{\mathcal{S}_{i}}=\frac{\left\vert
\begin{array}
[c]{l}%
\left(  \alpha^{2}f+u\beta^{2}f^{\prime}-2\varepsilon\mu u\left(
\varepsilon(\beta u)^{2}+(\alpha f)^{2}\right)  f^{\prime}\right)  \left(
f^{\prime2}-\varepsilon\right) \\
+\left(  \varepsilon(\beta u)^{2}+(\alpha f)^{2}\right)  \left(
2\varepsilon\lambda f\left(  f^{\prime2}-\varepsilon\right)  -f^{\prime\prime
}\right)
\end{array}
\right\vert }{2\left(  (\beta u)^{2}+\varepsilon(\alpha f)^{2}\right)  \left(
1-\varepsilon f^{\prime2}\right)  ^{3/2}}. \label{s34}%
\end{equation}
With similar procedure of the proof of the Theorem \ref{teo1}, from
(\ref{s34}) we can obtain the following theorem:

\begin{theorem}
The timelike general rotational surfaces of first and second types
$\mathcal{S}_{i}$ in $\mathbb{E}_{1}^{4}$ with density $e^{\lambda(x^{2}%
+y^{2})+\mu(z^{2}-t^{2})}$ are weighted minimal if and only if, up to
parametrization, the meridian curves are determined by $c_{1}:x_{1}%
(u)=(f(u),0,u,0)$ and $c_{2}:x_{2}(u)=(f(u),0,0,u),$ where $f(u)$ is a
solution to the following differential equations:

For $\mathcal{S}_{1}$
\begin{equation}
\left(  \arctan\left(  f^{\prime}(u)\right)  \right)  ^{\prime}=A_{2}
\label{s35}%
\end{equation}
and for $\mathcal{S}_{2}$
\begin{equation}
\left(  \log\left(  \frac{1-f^{\prime}(u)}{1+f^{\prime}(u)}\right)  \right)
^{\prime}=2A_{2}, \label{s36}%
\end{equation}
where
\begin{equation}
A_{2}=\frac{f(u)\left(  \alpha^{2}\varepsilon+2(\alpha^{2}\lambda
f(u)^{2}+\beta^{2}\varepsilon\lambda u^{2})\right)  +u\left(  \varepsilon
\beta^{2}\left(  1-2\mu u^{2}\right)  -2\alpha^{2}\mu f(u)^{2}\right)
f^{\prime}(u)}{(\beta u)^{2}+\varepsilon(\alpha f(u))^{2}}. \label{s37}%
\end{equation}

\end{theorem}

\subsection{Weighted Flat Timelike General Rotational surfaces in
$\mathbb{E}_{1}^{4}$ with Density $e^{\lambda_{1}x^{2}+\lambda_{2}%
y^{2}+\lambda_{3}z^{2}+\lambda_{4}t^{2}}$}

\ 

From (\ref{s11}) and $K_{\phi}^{\mathcal{S}_{i}}=K^{\mathcal{S}_{i}}%
-\Delta\phi$, we have

\begin{proposition}
The weighted Gaussian curvatures of the timelike general rotation surfaces of
first and second types $\mathcal{S}_{i}$ parametrized by (\ref{s1}) and
(\ref{s5}) in $\mathbb{E}_{1}^{4}$ with density $e^{\lambda_{1}x^{2}%
+\lambda_{2}y^{2}+\lambda_{3}z^{2}+\lambda_{4}t^{2}}$ are%
\begin{equation}
K_{\phi}^{\mathcal{S}_{i}}=\frac{\left(
\begin{array}
[c]{l}%
-2\delta\left(  (\alpha f)^{2}+\varepsilon(\beta g)^{2}\right)  ^{2}\left(
-\varepsilon f^{\prime2}+g^{\prime2}\right)  ^{2}+\alpha^{2}\beta^{2}\left(
gf^{\prime}-fg^{\prime}\right)  ^{2}\left(  -\varepsilon f^{\prime2}%
+g^{\prime2}\right) \\
+\left(  \varepsilon(\alpha f)^{2}+(\beta g)^{2}\right)  \left(  \beta
^{2}gf^{\prime}+\alpha^{2}fg^{\prime}\right)  \left(  f^{\prime\prime
}g^{\prime}-f^{\prime}g^{\prime\prime}\right)
\end{array}
\right)  }{\left(  (\alpha f)^{2}+\varepsilon(\beta g)^{2}\right)  ^{2}\left(
-\varepsilon f^{\prime2}+g^{\prime2}\right)  ^{2}}, \label{s39}%
\end{equation}
where $\delta=\lambda_{1}$$+\lambda$$_{2}+\lambda$$_{3}-\lambda$$_{4}.$
\end{proposition}

Consequently, we have

\begin{theorem}
The timelike general rotation surfaces of first and second types
$\mathcal{S}_{i}$ parametrized by (\ref{s1}) and (\ref{s5}) in $\mathbb{E}%
_{1}^{4}$ with density $e^{\lambda_{1}x^{2}+\lambda_{2}y^{2}+\lambda_{3}%
z^{2}+\lambda_{4}t^{2}}$ are weighted flat if and only if the equation%
\begin{align}
&  -2\delta\left(  (\alpha f)^{2}+\varepsilon(\beta g)^{2}\right)  ^{2}\left(
-\varepsilon f^{\prime2}+g^{\prime2}\right)  ^{2}+\alpha^{2}\beta^{2}\left(
gf^{\prime}-fg^{\prime}\right)  ^{2}\left(  -\varepsilon f^{\prime2}%
+g^{\prime2}\right) \nonumber\\
&  +\left(  \varepsilon(\alpha f)^{2}+(\beta g)^{2}\right)  \left(  \beta
^{2}gf^{\prime}+\alpha^{2}fg^{\prime}\right)  \left(  f^{\prime\prime
}g^{\prime}-f^{\prime}g^{\prime\prime}\right)  =0 \label{s40}%
\end{align}
holds.
\end{theorem}

\begin{corollary}
For $\lambda_{1}+\lambda_{2}+\lambda_{3}=\lambda_{4},$ the timelike general
rotation surfaces of first and second types $\mathcal{S}_{i}$ parametrized by
(\ref{s1}) and (\ref{s5}) in $\mathbb{E}_{1}^{4}$ with density $e^{\lambda
_{1}x^{2}+\lambda_{2}y^{2}+\lambda_{3}z^{2}+\lambda_{4}t^{2}}$ are weighted
flat if the equation
\begin{equation}
\alpha^{2}\beta^{2}\left(  gf^{\prime}-fg^{\prime}\right)  ^{2}\left(
-\varepsilon f^{\prime2}+g^{\prime2}\right)  +\left(  \varepsilon(\alpha
f)^{2}+(\beta g)^{2}\right)  \left(  \beta^{2}gf^{\prime}+\alpha^{2}%
fg^{\prime}\right)  \left(  f^{\prime\prime}g^{\prime}-f^{\prime}%
g^{\prime\prime}\right)  =0 \label{s41}%
\end{equation}
holds.
\end{corollary}

Thus, we have

\begin{theorem}
For $\lambda_{1}+\lambda_{2}+\lambda_{3}=\lambda_{4},$ the timelike general
rotation surfaces of first and second types $\mathcal{S}_{i}$ parametrized by
(\ref{s1}) and (\ref{s5}) in $\mathbb{E}_{1}^{4}$ with density $e^{\lambda
_{1}x^{2}+\lambda_{2}y^{2}+\lambda_{3}z^{2}+\lambda_{4}t^{2}}$ are weighted
flat if the functions $f$ and $g$ are linear dependent.
\end{theorem}

Now, let we construct two examples for the weighted flat timelike general
rotational surfaces of first and second types $\mathcal{S}_{i}.$

\begin{example}
If we take $f(u)=u^{3},$ $g(u)=2u^{3}$ and $\alpha=1,$ $\beta=2$ in
(\ref{s1}), then we get the following timelike general rotational surfaces of
first type:%
\begin{equation}
\mathcal{S}_{1}:X(u,v)=(u^{3}\cos v,u^{3}\sin v,2u^{3}\cosh(2v),2u^{3}%
\sinh(2v)). \label{ex7}%
\end{equation}
The weighted Gaussian curvature of this surface in $\mathbb{E}_{1}^{4}$ with
density $e^{\lambda_{1}x^{2}+\lambda_{2}y^{2}+\lambda_{3}z^{2}+\lambda
_{4}t^{2}}$ is%
\begin{equation}
K_{\phi}^{\mathcal{S}_{1}}=-2\delta=-2(\lambda_{1}+\lambda_{2}+\lambda
_{3}-\lambda_{4}) \label{ex8}%
\end{equation}
Thus, the timelike general rotational surfaces of first type $\mathcal{S}_{1}$
parametrized by (\ref{ex7}) is weighted flat in $\mathbb{E}_{1}^{4}$ with
density $e^{\lambda_{1}x^{2}+\lambda_{2}y^{2}+\lambda_{3}z^{2}+\lambda
_{4}t^{2}},$ where $\lambda_{1}+\lambda_{2}+\lambda_{3}=\lambda_{4}.$ In
Figure 2(a), one can see the projection of the surface (\ref{ex7}) into
$x_{1}x_{2}x_{3}$-space.
\end{example}

\begin{example}
If we take $f(u)=e^{2u},$ $g(u)=3e^{2u}$ and $\alpha=3,$ $\beta=2$ in
(\ref{s5}), then we get the following timelike general rotational surfaces of
second type:%
\begin{equation}
\mathcal{S}_{2}:X(u,v)=(e^{2u}\cos(3v),e^{2u}\sin(3v),3e^{2u}\sinh
(2v),3e^{2u}\cosh(2v)). \label{ex10}%
\end{equation}
The weighted Gaussian curvature of this surface in $\mathbb{E}_{1}^{4}$ with
density $e^{\lambda_{1}x^{2}+\lambda_{2}y^{2}+\lambda_{3}z^{2}+\lambda
_{4}t^{2}}$ is%
\begin{equation}
K_{\phi}^{\mathcal{S}_{2}}=-2\delta=-2(\lambda_{1}+\lambda_{2}+\lambda
_{3}-\lambda_{4}) \label{ex11}%
\end{equation}
Thus, the timelike general rotational surfaces of second type $\mathcal{S}%
_{2}$ parametrized by (\ref{ex10}) is weighted flat in $\mathbb{E}_{1}^{4}$
with density $e^{\lambda_{1}x^{2}+\lambda_{2}y^{2}+\lambda_{3}z^{2}%
+\lambda_{4}t^{2}},$ where $\lambda_{1}+\lambda_{2}+\lambda_{3}=\lambda_{4}.$
In Figure 2(b), one can see the projection of the surface (\ref{ex10}) into
$x_{1}x_{2}x_{4}$-space.
\end{example}

\begin{figure}[H]
\centering
\includegraphics[
	height=3.5in, width=5.1in
	]	{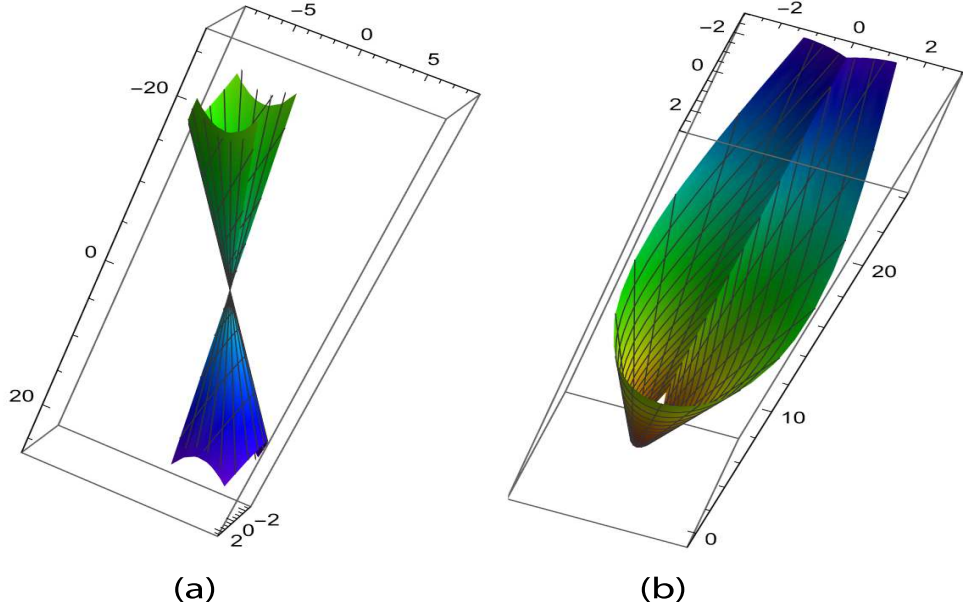}\caption{The projection of the weighted flat timelike general
rotational surface of first type (\ref{ex7}) into $x_{1}x_{2}x_{3}$-space
\textbf{(a)} and the projection of the weighted flat timelike general
rotational surface of second type (\ref{ex10}) into $x_{1}x_{2}x_{4}$-space
\textbf{(b)}}%
\label{fig:2}%
\end{figure}

Now, let we suppose that $f(u)=k$ for the timelike general rotation surfaces
of first and second types $\mathcal{S}_{i}$ (resp. $g(u)=l$ for the timelike
general rotation surface of first type $\mathcal{S}_{1}$) $k,l$ are a nonzero
constants. Then from (\ref{s39}), the weighted Gaussian curvatures are%
\begin{align}
&  K_{\phi}^{\mathcal{S}_{i}}=\frac{\alpha^{2}\beta^{2}k^{2}-2\delta\left(
(\alpha k)^{2}+\varepsilon(\beta g)^{2}\right)  ^{2}}{\left(  (\alpha
k)^{2}+\varepsilon(\beta g)^{2}\right)  ^{2}}\label{s42}\\
&  \left(  \text{ resp. }K_{\phi}^{\mathcal{S}_{1}}=\frac{\alpha^{2}\beta
^{2}l^{2}-2\delta\left(  (\beta l)^{2}-(\alpha f)^{2}\right)  ^{2}}{\left(
(\beta l)^{2}-(\alpha f)^{2}\right)  ^{2}}\right)  .\label{s43}%
\end{align}
If we find the function $g(u)$ (resp. $f(u)$) from (\ref{s42}) (resp.
(\ref{s43})) according to the weighted Gaussian curvature, then we get%
\begin{align}
&  g(u)=\mp\frac{1}{\beta}\sqrt{-\varepsilon(\alpha k)^{2}\mp\frac{\alpha\beta
k}{\sqrt{K_{\phi}^{\mathcal{S}_{i}}+2\delta}}}\label{s44}\\
&  \left(  \text{ resp. }f(u)=\mp\frac{1}{\alpha}\sqrt{(\beta l)^{2}\mp
\frac{\alpha\beta l}{\sqrt{K_{\phi}^{\mathcal{S}_{1}}+2\delta}}}\right)
.\label{s45}%
\end{align}
Consequently, we have

\begin{theorem}
The timelike general rotation surfaces of first and second types
$\mathcal{S}_{i}$ parametrized by (\ref{s1}) and (\ref{s5}) in $\mathbb{E}%
_{1}^{4}$ with density $e^{\lambda_{1}x^{2}+\lambda_{2}y^{2}+\lambda_{3}%
z^{2}+\lambda_{4}t^{2}}$ can be parametrized in terms of the weighted Gaussian
curvature by%
\begin{align}
\mathcal{S}_{1}  &  :X_{1}(u,v)=(k\cos(\alpha v),k\sin(\alpha v),\mp\frac
{1}{\beta}\sqrt{(\alpha k)^{2}\mp\frac{\alpha\beta k}{\sqrt{K_{\phi
}^{\mathcal{S}_{1}}+2\delta}}}\cosh(\beta v),\nonumber\\
&  \text{
\ \ \ \ \ \ \ \ \ \ \ \ \ \ \ \ \ \ \ \ \ \ \ \ \ \ \ \ \ \ \ \ \ \ \ \ \ \ \ \ }%
\mp\frac{1}{\beta}\sqrt{(\alpha k)^{2}\mp\frac{\alpha\beta k}{\sqrt{K_{\phi
}^{\mathcal{S}_{1}}+2\delta}}}\sinh(\beta v)),\label{s46}\\
\mathcal{S}_{2}  &  :X_{2}(u,v)=(k\cos(\alpha v),k\sin(\alpha v),\mp\frac
{1}{\beta}\sqrt{-(\alpha k)^{2}\mp\frac{\alpha\beta k}{\sqrt{K_{\phi
}^{\mathcal{S}_{2}}+2\delta}}}\sinh(\beta v),\nonumber\\
&  \text{
\ \ \ \ \ \ \ \ \ \ \ \ \ \ \ \ \ \ \ \ \ \ \ \ \ \ \ \ \ \ \ \ \ \ \ \ \ \ \ \ }%
\mp\frac{1}{\beta}\sqrt{-(\alpha k)^{2}\mp\frac{\alpha\beta k}{\sqrt{K_{\phi
}^{\mathcal{S}_{2}}+2\delta}}}\cosh(\beta v)). \label{s47}%
\end{align}%
\begin{equation}
\left(
\begin{array}
[c]{l}%
\text{resp. }\mathcal{S}_{1}:X_{1}(u,v)=(\mp\frac{1}{\alpha}\sqrt{(\beta
l)^{2}\mp\frac{\alpha\beta l}{\sqrt{K_{\phi}^{\mathcal{S}_{1}}+2\delta}}}%
\cos(\alpha v),\mp\frac{1}{\alpha}\sqrt{(\beta l)^{2}\mp\frac{\alpha\beta
l}{\sqrt{K_{\phi}^{\mathcal{S}_{1}}+2\delta}}}\sin(\alpha v),\\
\text{\ \ \ \ \ \ \ \ \ \ \ \ \ \ \ \ \ \ \ \ \ \ \ \ \ \ \ }l\cosh(\beta
v),l\sinh(\beta v)),
\end{array}
\right)  \label{s48}%
\end{equation}
where $\alpha,$ $\beta,$ $k,$ $l$ and $\lambda_{i}$ are nonzero constants
($\alpha>0,$ $\beta>0$).
\end{theorem}

\begin{corollary}
The weighted Gaussian curvatures of the timelike general rotational surfaces
of first and second types $\mathcal{S}_{i}$ (resp. the weighted Gaussian
curvatures of the timelike general rotational surface of first type
$\mathcal{S}_{1}$) in $\mathbb{E}_{1}^{4}$ with density $e^{\lambda_{1}%
x^{2}+\lambda_{2}y^{2}+\lambda_{3}z^{2}+\lambda_{4}t^{2}}$ cannot be constant
when $f(u)$ (resp. $g(u)$) is a nonzero constant.
\end{corollary}

Now, let us suppose that $f(u)=u$. In this case, from (\ref{s39}), the
weighted Gaussian curvatures are%
\begin{equation}
K_{\phi}^{\mathcal{S}_{i}}=\frac{\left(
\begin{array}
[c]{l}%
\alpha^{2}\beta^{2}\left(  g-ug^{\prime}\right)  ^{2}\left(  g^{\prime
2}-\varepsilon\right)  -\varepsilon g^{\prime\prime}\left(  (\alpha
u)^{2}+\varepsilon(\beta g)^{2}\right)  \left(  \alpha^{2}ug^{\prime}%
+\beta^{2}g\right) \\
-2\delta\left(  g^{\prime2}-\varepsilon\right)  ^{2}\left(  (\alpha
u)^{2}+\varepsilon(\beta g)^{2}\right)  ^{2}%
\end{array}
\right)  }{\left(  g^{\prime2}-\varepsilon\right)  ^{2}\left(  (\alpha
u)^{2}+\varepsilon(\beta g)^{2}\right)  ^{2}}. \label{s49}%
\end{equation}

So, we can state the following corollary:

\begin{theorem}
\label{teo2}The timelike general rotational surfaces of first and second types
$\mathcal{S}_{i}$ in $\mathbb{E}_{1}^{4}$ with density $e^{\lambda_{1}%
x^{2}+\lambda_{2}y^{2}+\lambda_{3}z^{2}+\lambda_{4}t^{2}}$ are weighted flat
if and only if, up to parametrization, the meridian curves are determined by
$c_{1}:x_{1}(u)=(u,0,g(u),0)$ and $c_{2}:x_{2}(u)=(u,0,0,g(u)),$ where $g(u)$
is a solution to the following differential equations:

For $\mathcal{S}_{1}$
\begin{equation}
\left(  \arctan\left(  g^{\prime}(u)\right)  \right)  ^{\prime}=A_{3}
\label{s50}%
\end{equation}
and for $\mathcal{S}_{2}$
\begin{equation}
\left(  \log\left(  \frac{1-g^{\prime}(u)}{1+g^{\prime}(u)}\right)  \right)
^{\prime}=2A_{3}, \label{s51}%
\end{equation}
where
\begin{equation}
A_{3}=\frac{-2\delta\left(  (\alpha u)^{2}+\varepsilon(\beta g(u))^{2}\right)
^{2}\left(  -\varepsilon+g^{\prime}(u)^{2}\right)  +\alpha^{2}\beta^{2}\left(
g(u)-ug^{\prime}(u)\right)  ^{2}}{\left(  \varepsilon(\alpha u)^{2}+(\beta
g(u))^{2}\right)  \left(  \beta^{2}g(u)+\alpha^{2}ug^{\prime}(u)\right)  }.
\label{s52}%
\end{equation}

\end{theorem}

\begin{proof}
If the timelike general rotational surfaces of first and second types
$\mathcal{S}_{i}$ in $\mathbb{E}_{1}^{4}$ with density $e^{\lambda_{1}%
x^{2}+\lambda_{2}y^{2}+\lambda_{3}z^{2}+\lambda_{4}t^{2}}$ are weighted flat,
then from (\ref{s49}) we have%
\begin{equation}
\frac{g^{\prime\prime}(u)}{-\varepsilon+g^{\prime}(u)^{2}}=\frac
{-2\delta\left(  (\alpha u)^{2}+\varepsilon(\beta g(u))^{2}\right)
^{2}\left(  -\varepsilon+g^{\prime}(u)^{2}\right)  +\alpha^{2}\beta^{2}\left(
g(u)-ug^{\prime}(u)\right)  ^{2}}{\left(  \varepsilon(\alpha u)^{2}+(\beta
g(u))^{2}\right)  \left(  \beta^{2}g(u)+\alpha^{2}ug^{\prime}(u)\right)  }.
\label{s52'}%
\end{equation}
Thus from (\ref{s52'}), we get (\ref{s50}) and (\ref{s51}) for $\varepsilon
=-1$ and $\varepsilon=1,$ respectively.
\end{proof}

Now, let us suppose that $g(u)=u$. In this case, from (\ref{s39}), the
weighted Gaussian curvatures are%
\begin{equation}
K_{\phi}^{\mathcal{S}_{i}}=\frac{\left(
\begin{array}
[c]{l}%
-\varepsilon\alpha^{2}\beta^{2}\left(  f-uf^{\prime}\right)  ^{2}\left(
-\varepsilon+f^{\prime2}\right)  +\left(  (\beta u)^{2}+\varepsilon(\alpha
f)^{2}\right)  \left(  \alpha^{2}f+u\beta^{2}f^{\prime}\right)  f^{\prime
\prime}\\
-2\delta\left(  (\beta u)^{2}+\varepsilon(\alpha f)^{2}\right)  ^{2}\left(
f^{\prime2}-\varepsilon\right)  ^{2}%
\end{array}
\right)  }{\left(  (\beta u)^{2}+\varepsilon(\alpha f)^{2}\right)  ^{2}\left(
f^{\prime2}-\varepsilon\right)  ^{2}}. \label{s53}%
\end{equation}
So with similar procedure of the proof of the Theorem \ref{teo2}, from
(\ref{s53}) we can obtain the following theorem:

\begin{theorem}
The timelike general rotational surfaces of first and second types
$\mathcal{S}_{i}$ in $\mathbb{E}_{1}^{4}$ with density $e^{\lambda_{1}%
x^{2}+\lambda_{2}y^{2}+\lambda_{3}z^{2}+\lambda_{4}t^{2}}$ are weighted flat
if and only if, up to parametrization, the meridian curves are determined by
$c_{1}:x_{1}(u)=(f(u),0,u,0)$ and $c_{2}:x_{2}(u)=(f(u),0,0,u),$ where $f(u)$
is a solution to the following differential equations:

For $\mathcal{S}_{1}$
\begin{equation}
\left(  \arctan\left(  f^{\prime}(u)\right)  \right)  ^{\prime}=A_{4}
\label{s54}%
\end{equation}
and
\begin{equation}
\left(  \log\left(  \frac{1-f^{\prime}(u)}{1+f^{\prime}(u)}\right)  \right)
^{\prime}=2A_{4}, \label{s55}%
\end{equation}
where
\begin{equation}
A_{4}=-\frac{-2\delta\left(  (\beta u)^{2}+\varepsilon(\alpha f(u))^{2}%
\right)  ^{2}\left(  -\varepsilon+f^{\prime}(u)^{2}\right)  -\varepsilon
\alpha^{2}\beta^{2}\left(  f(u)-uf^{\prime}(u)\right)  ^{2}}{\left(  (\beta
u)^{2}+\varepsilon(\alpha f(u))^{2}\right)  \left(  \alpha^{2}f(u)+u\beta
^{2}f^{\prime}(u)\right)  }. \label{s56}%
\end{equation}

\end{theorem}

Now, let we construct two examples for the timelike general rotational
surfaces of first and second types $\mathcal{S}_{i}$ and find their weighted
mean and weighted Gaussian curvatures for different densities.

\begin{example}
If we take $f(u)=\tan u\ $and $g(u)=\sec u$ in (\ref{s1}), then we get the
following general rotational surfaces of first type:%
\begin{equation}
\mathcal{S}_{1}:X(u,v)=(\tan u\cos(\alpha v),\tan u\sin(\alpha v),\sec
u\cosh(\beta v),\sec u\sinh(\beta v)). \label{ex13}%
\end{equation}
The weighted mean and weighted Gaussian curvatures of this surface in
$\mathbb{E}_{1}^{4}$ with density $e^{\lambda_{1}x^{2}+\lambda_{2}%
y^{2}+\lambda_{3}z^{2}+\lambda_{4}t^{2}}$ is%
\begin{equation}
H_{\phi}^{\mathcal{S}_{1}}=\sqrt{\frac{\left(
\begin{array}
[c]{l}%
\left(
\begin{array}
[c]{l}%
\frac{\left(  \alpha^{2}\tan^{2}u-\beta^{2}\sec^{2}u\right)  \sec^{3}u}%
{2}\left(
\begin{array}
[c]{l}%
{\small (}\cos{\small (2u)-3)}\left(
\begin{array}
[c]{l}%
\text{$\lambda_{1}$}{\small +}\text{$\lambda$}_{2}\\
{\small +(}\text{$\lambda$}_{1}{\small -}\text{$\lambda$}_{2}{\small )}%
\cos{\small (2\alpha v)}%
\end{array}
\right)  \tan^{2}{\small u}\sec^{2}{\small u}\\
{\small -4}\tan^{2}{\small u+2}\left(  \tan^{2}u+\sec^{2}u\right)
\end{array}
\right) \\
+\left(  \tan^{2}u+\sec^{2}u\right)  \left(  \alpha^{2}\tan^{2}u-\beta^{2}%
\sec^{2}u\right)  \left(  \text{$\lambda$}_{3}{\small -}\text{$\lambda_{4}$%
}{\small +(}\text{$\lambda$}_{3}{\small +}\text{$\lambda$}_{4}{\small )}%
\cosh{\small (2\beta v)}\right)  \sec^{5}u\\
+\left(  \tan^{2}u+\sec^{2}u\right)  \left(  \alpha^{2}\tan^{2}u\sec
u+\beta^{2}\sec^{3}u\right)  \sec^{2}u
\end{array}
\right)  ^{2}\\
{\small +(}\tan^{2}{\small u}\sec^{8}{\small u}\left(  \beta^{2}\sec
^{2}u-\alpha^{2}\tan^{2}u\right)  {\small (\beta(}\text{$\lambda$}%
_{2}{\small -}\text{$\lambda$}_{1}{\small )}\sin{\small (2\alpha v)}\\
{\small +\alpha(}\text{$\lambda$}_{3}{\small +}\text{$\lambda$}_{4}%
{\small )}\sinh{\small (2\beta v))}^{2}{\small )}\left(  \tan^{2}u+\sec
^{2}u\right)  ^{3}%
\end{array}
\right)  }{4\left(  \tan^{2}u+\sec^{2}u\right)  ^{3}\left(  \beta^{2}\sec
^{2}u-\alpha^{2}\tan^{2}u\right)  ^{2}\cos^{-6}u}} \label{ex14}%
\end{equation}
and%
\begin{equation}
K_{\phi}^{\mathcal{S}_{1}}=\frac{\left(  \alpha^{2}-\beta^{2}\right)  \left(
\alpha^{2}\tan^{4}(u)+\beta^{2}\sec^{4}(u)\right)  }{\left(  \tan^{2}%
(u)+\sec^{2}(u)\right)  ^{2}\left(  \beta^{2}\sec^{2}(u)-\alpha^{2}\tan
^{2}(u)\right)  ^{2}}-2\delta, \label{ex15}%
\end{equation}
respectively. If we take $\alpha=\beta$ in (\ref{ex14}) and (\ref{ex15})$,$
then we get the weighted mean and weighted Gaussian curvatures of timelike
general rotational surfaces of first type (\ref{ex13}) as%
\begin{equation}
H_{\phi}^{\mathcal{S}_{1}}=\frac{\left\vert (7\lambda-12\mu+4(\mu
-2(\lambda+1))\cos(2u)+\lambda\cos(4u)+8)\sec u\right\vert }{2\sqrt
{2(3-\cos(2u))^{3}}} \label{ex17}%
\end{equation}
and%
\begin{equation}
K_{\phi}^{\mathcal{S}_{1}}=-4(\lambda+\mu) \label{ex18}%
\end{equation}
in $\mathbb{E}_{1}^{4}$ with density $e^{\lambda(x^{2}+y^{2})+\mu(z^{2}%
-t^{2})}.$ In Figure 3(a), one can see the projection of the surface
(\ref{ex13}) for $\alpha=\beta=3$ into $x_{1}x_{2}x_{3}$-space.
\end{example}

\begin{example}
If we take $f(u)=\sec u\ $and $g(u)=\tan u$ in (\ref{s5}), then we get the
following timelike general rotational surfaces of second type:%
\begin{equation}
\mathcal{S}_{2}:X(u,v)=(\sec u\cos(\alpha v),\sec u\sin(\alpha v),\tan
u\sinh(\beta v),\tan u\cosh(\beta v)). \label{ex19}%
\end{equation}
The weighted mean and weighted Gaussian curvatures of this surface in
$\mathbb{E}_{1}^{4}$ with density $e^{\lambda_{1}x^{2}+\lambda_{2}%
y^{2}+\lambda_{3}z^{2}+\lambda_{4}t^{2}}$ is%
\begin{equation}
H_{\phi}^{\mathcal{S}_{2}}=\frac{1}{2}\sqrt{%
\begin{array}
[c]{l}%
\frac{\left(  \left(  \beta(\text{$\lambda$}_{1}-\text{$\lambda$}_{2}%
)\sin(2\alpha v)+\alpha(\text{$\lambda$}_{3}+\text{$\lambda$}_{4})\sinh(2\beta
v)\right)  ^{2}\sin^{2}u\right)  }{\cos^{4}u(\alpha^{2}\sec^{2}u+\beta^{2}%
\tan^{2}u)}+\\
\left(
\begin{array}
[c]{l}%
\left(  -\text{$\lambda$}_{3}+\text{$\lambda$}_{4}\right)  \tan^{2}u+\left(
\text{$\lambda$}_{1}+\text{$\lambda$}_{2}+(\text{$\lambda$}_{1}-\text{$\lambda
$}_{2})\cos(2\alpha v)\right)  \sec^{2}u\\
+(\text{$\lambda$}_{3}+\text{$\lambda$}_{4})\tan^{2}u\cosh(2\beta v)+2
\end{array}
\right)  ^{2}%
\end{array}
} \label{exy1}%
\end{equation}
and
\begin{equation}
K_{\phi}^{\mathcal{S}_{2}}=\frac{4\left(  \alpha^{2}+\beta^{2}\right)  \left(
\alpha^{2}+\beta^{2}\sin^{4}u\right)  -2\delta\left(  2\alpha^{2}+\beta
^{2}\left(  1-\cos(2u)\right)  \right)  ^{2}}{\left(  2\alpha^{2}+\beta
^{2}\left(  1-\cos(2u)\right)  \right)  ^{2}} \label{exy2}%
\end{equation}
respectively. If we take $\alpha=\beta$ in (\ref{exy1}) and (\ref{exy2}) then
we get the weighted mean and weighted Gaussian curvatures of timelike general
rotational surfaces of second type (\ref{ex19}) as%
\begin{equation}
H_{\phi}^{\mathcal{S}_{2}}=\frac{\left\vert ((\mu+1)\cos(2u)+2\lambda
-\mu+1)\right\vert }{2\cos^{2}u} \label{ex20}%
\end{equation}
and%
\begin{equation}
K_{\phi}^{\mathcal{S}_{2}}=\frac{4(6(\lambda+\mu)-1)\cos(2u)+(1-2(\lambda
+\mu))\cos(4u)-38\left(  \lambda+\mu\right)  +11}{(\cos(2u)-3)^{2}}
\label{ex21}%
\end{equation}
in $\mathbb{E}_{1}^{4}$ with density $e^{\lambda(x^{2}+y^{2})+\mu(z^{2}%
-t^{2})}.$ In Figure 3(b), one can see the projection of the surface
(\ref{ex19}) for $\alpha=\beta=3$ into $x_{1}x_{2}x_{3}$-space.
\end{example}

\begin{figure}[H]
\centering
\includegraphics[
	height=3.3in, width=5.6in
	]	{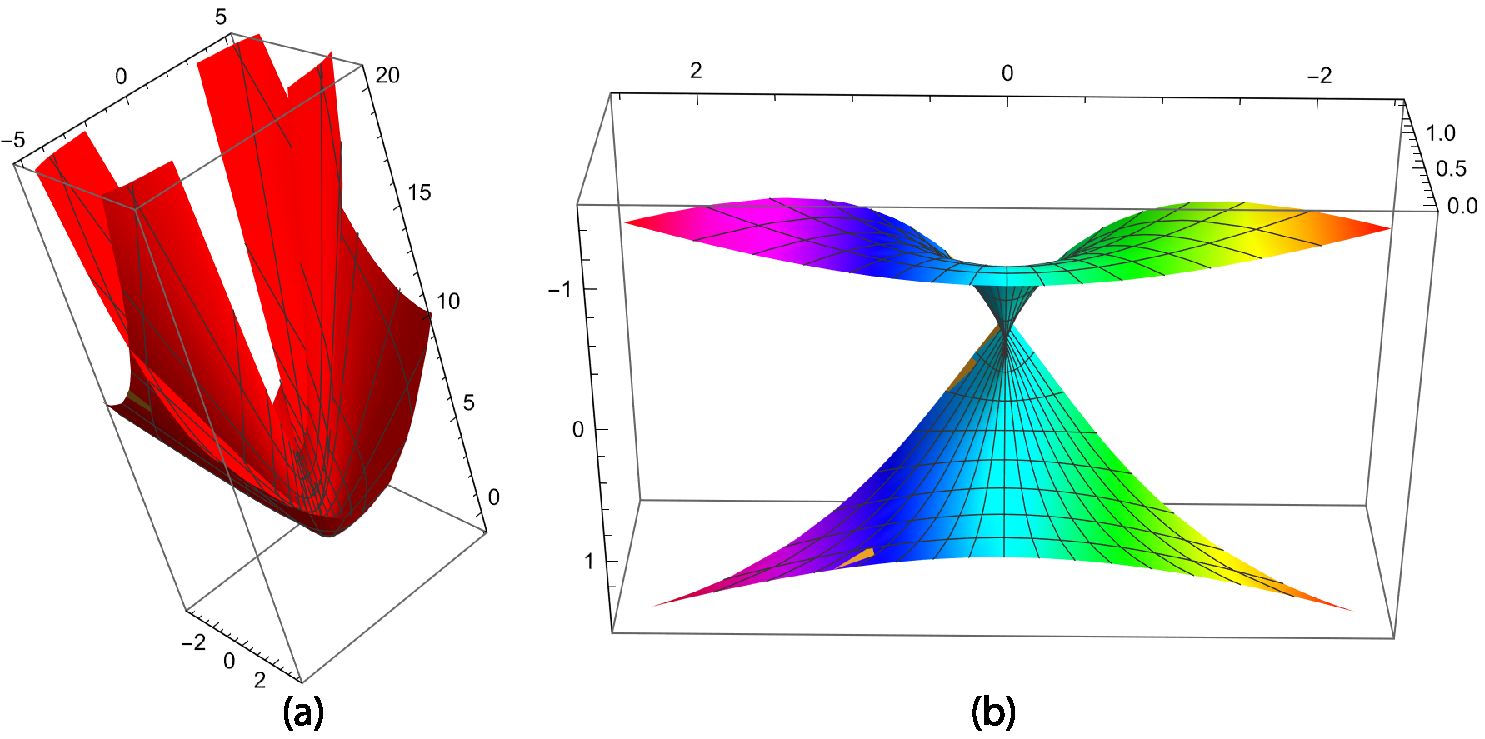}\caption{The projection of the timelike general rotational
surface of first type (\ref{ex13}) into $x_{1}x_{2}x_{3}$-space \textbf{(a)}
and the projection of the timelike general rotational surface of second type
(\ref{ex19}) into $x_{1}x_{2}x_{3}$-space \textbf{(b)}}%
\label{fig:3}%
\end{figure}

\end{document}